\documentclass[12pt]{amsart}
\usepackage{a4wide}
\usepackage[T1]{fontenc}
\usepackage{amssymb,amsmath,amsthm,latexsym}
\usepackage{mathrsfs}
\usepackage[usenames,dvipsnames]{color}
\usepackage{euscript}
\usepackage{graphicx}
\usepackage{mdwlist}
\usepackage{enumerate}
\usepackage{mathtools,dsfont,wasysym}
\usepackage{stmaryrd}
\usepackage{tikz-cd}
\usepackage{enumitem}
\usepackage{hyperref}
\hypersetup{colorlinks=true,  linkcolor=blue, citecolor=red, urlcolor=cyan}

\newtheorem{theorem}{Theorem}[section]
\newtheorem{lemma}[theorem]{Lemma}
\newtheorem{corollary}[theorem]{Corollary}
\newtheorem{proposition}[theorem]{Proposition}

\theoremstyle{remark}
\newtheorem{remark}[theorem]{Remark}
\theoremstyle{definition}
\newtheorem{definition}[theorem]{Definition}

\numberwithin{equation}{section}
\makeatother

\newcommand{\R}{\mathbb{R}}
\renewcommand{\H}{\mathbb{H}}
\newcommand{\N}{\mathbb{N}}
\newcommand{\e}{\varepsilon}
\newcommand{\p}{\varphi}
\newcommand{\F}{\mathcal{F}}
\newcommand{\id}{\mathrm{id}}
\newcommand{\Lip}{\mathrm{Lip}}
\renewcommand{\leq}{\leqslant}
\renewcommand{\geq}{\geqslant}
\newcommand{\arcosh}{\mathrm{arcosh}}
\newcommand{\diam}{\operatorname{diam}}
\newcommand{\dist}{\operatorname{dist}}
\newcommand{\wtow}{weak$^*$-to-weak$^*$ }
\newcommand*{\mathbbb}{\text{\usefont{U}{bbold}{m}{n}1}}
 
\begin{document}
\title[Tilings of the hyperbolic space and Lipschitz functions]{Tilings of the hyperbolic space\\and Lipschitz functions}
\author{Christian Bargetz \and Franz Luggin \and Tommaso Russo}
\address{Universit\"{a}t Innsbruck, Department of Mathematics, Technikerstra\ss e 13, 6020 Innsbruck, Austria}
\email{Christian.Bargetz@uibk.ac.at}
\email{Franz.Luggin@student.uibk.ac.at}
\email{Tommaso.Russo@uibk.ac.at, tommaso.russo.math@gmail.com}

\keywords{Lipschitz-free space, Banach space of Lipschitz functions, Hyperbolic $d$-space, tiling, polytope, linear extension operator, Schauder basis}
\subjclass[2020]{46B03, 51M10 (primary), and 46B20, 46E15, 26B35, 20F55 (secondary)}
\date{\today}

\begin{abstract} We use a special tiling for the hyperbolic $d$-space $\mathbb{H}^d$ for $d=2,3,4$ to construct an (almost) explicit isomorphism between the Lipschitz-free space $\mathcal{F}(\mathbb{H}^d)$ and $\mathcal{F}(P)\oplus \mathcal{F}(\mathcal{N})$ where $P$ is a polytope in $\mathbb{R}^d$ and $\mathcal{N}$ a net in $\mathbb{H}^d$ coming from the tiling. This implies that the spaces $\mathcal{F}(\mathbb{H}^d)$ and $\mathcal{F}(\mathbb{R}^d)\oplus \mathcal{F}(\mathcal{M})$ are isomorphic for every net $\mathcal{M}$ in $\mathbb{H}^d$. In particular, we obtain that, for $d=2,3,4$, $\mathcal{F}(\mathbb{H}^d)$ has a Schauder basis. Moreover, using a similar method, we also give an explicit isomorphism between ${\rm Lip}(\mathbb{H}^d)$ and ${\rm Lip}(\mathbb{R}^d)$.
\end{abstract}
\maketitle

\section{Introduction}
Given a metric space $M$ with a distinguished point $0_M\in M$, the Lipschitz-free space $\mathcal{F}(M)$, together with an isometric mapping $\delta\colon M \to \mathcal{F}(M)$, is the uniquely determined (up to linear isometry) Banach space with the following universal property: for every Lipschitz mapping $f\colon M \to X$ to a Banach space $X$ with $f(0_M)=0$ there is a unique bounded linear operator $F\colon \mathcal{F}(M) \to X$ with $\|F\| = \Lip (f)$ such that the diagram
\begin{center}
\begin{tikzcd}
    \mathcal{F}(M) \arrow{r}{F} & X \\ M \arrow{u}{\delta} \arrow{ru}{f}
\end{tikzcd}
\end{center}
commutes. See, for instance, \cite{GHT-free} for an approach to Lipschitz-free spaces via the universal property. The dual space of $\F(M)$ is the space $\Lip_0(M)$ of Lipschitz functions $f\colon M \to \mathbb{R}$ with $f(0_M)=0$ equipped with the Lipschitz constant as norm, i.e. $\|f\|:=\Lip (f)$. Note that the condition $f(0_M)=0$ eliminates the constant functions and hence ensures that the Lipschitz constant is indeed a norm. 

The name Lipschitz-free spaces was introduced by G.~Godefroy and N.J.~Kalton in~\cite{GK2003} where, among others, these spaces are used to construct canonical examples of non-separable Banach spaces which are bi-Lipschitz equivalent but not linearly isomorphic. Such spaces have been studied by several authors in different contexts and with different terminology and we refer to \cite{Weaver} and \cite[Section 1.6]{OO19} for some terminological and historical remarks. The appearance of \cite{GK2003} resulted in a new impetus to their study, in particular in connection with non-linear functional analysis, metric geometry, and theoretical computer science; let us refer to \cite{AACD, AACD-bounded, AGPP2022, AP3-support, DK2022, FG2022, Veeorg, Weaver-predual} for a non-exhaustive list of some recent results. Let us also refer to \cite{Godefroy-survey} for a recent nice survey on some aspects of the theory of Lipschitz-free spaces. A detailed exposition of the spaces $\Lip_0(M)$ and $\F(M)$ can also be found in N.~Weaver's book~\cite{Weaver}. There, in addition to the Banach space $\Lip_0(M)$, also the Banach space $\Lip(M)$ of bounded real-valued Lipschitz functions is introduced and studied in detail. This space has the additional important property to be a Banach algebra.
\smallskip

In metric geometry, up to dimension and a scaling of the metric, there only exist three model spaces (see e.g. \cite[Chapter~I.2]{BH}): the sphere $\mathbb{S}^d$ with the intrinsic (geodesic) metric, the Euclidean space $\R^d$, and the hyperbolic space $\H^d$. The structure of $\F(\R^d)$ is well-studied, both from the isometric \cite{CKK2016} and from the isomorphic point of view \cite{AACD-Schauder, BM2012, CDW2016, HP2014, PS2015}. Quite recently it was proved that the spaces $\F(\mathbb{S}^d)$ and $\F(\R^d)$ are isomorphic, \cite[Theorem 4.21]{AACD} (see also \cite{FG2022} for a more general result). On the other hand, much less seems to be known about the structure of the space $\F(\H^d)$ and its dual space $\Lip_0(\H^d)$; some results concerning $\F(\H^d)$ can be found in~\cite{DK2022} where, among others, the authors pose the questions of whether $\mathcal{F}(\mathbb{H}^{d})$ has a Schauder basis and whether it is isomorphic to $\mathcal{F}(\mathbb{R}^{d})$. While our work was under review, the preprint \cite{Gartland} by C.~Gartland actually gave a positive answer to the latter question.

The aim of this paper is to explain how the local structure of $\mathbb{H}^d$ for $d=2,3,4$ together with a macroscopic view of $\mathbb{H}^d$ determine the Banach space structure of $\mathcal{F}(\mathbb{H}^d)$ and $\Lip_0(\mathbb{H}^d)$. Since the space $\Lip_0(\mathbb{H}^d)$ is more tangible than $\mathcal{F}(\mathbb{H}^d)$, it will be more convenient for us to work with the space $\Lip_0(\mathbb{H}^d)$ and then transfer the results to the predual. More precisely, we construct an (almost) explicit isomorphism 
\[
\Phi\colon \Lip_0(\mathbb{H}^d) \simeq \Lip_0(P) \oplus \Lip_0(\mathcal{N})
\]
where $P$ is a polytope (with nonempty interior) in $\mathbb{R}^d$ and $\mathcal{N}$ is a suitable net in~$\mathbb{H}^d$. In particular, we build an explicit isomorphism between $\Lip_0(\H^d)$ and $Z \oplus \Lip_0(\mathcal{N})$, where $Z$ is a direct sum of certain subspaces of $\Lip_0(P)$. This gives us an explicit procedure to reduce the study of $\Lip_0(\H^d)$ to the discrete case of $\Lip_0(\mathcal{N})$ and to a space of Lipschitz functions on $\R^d$. Since our main focus is on the hyperbolic case, we allow ourselves to use non-explicit arguments, such as Lee--Naor extension results \cite{LeeNaor} and Pe{\l}czy\'{n}ski decomposition method, in the proof that $Z$ is isomorphic to $\Lip_0(P)$ (in Section~\ref{sec:local result}). Let us however point out that this part of the argument could also be made entirely explicit, by using a variant of the arguments from Section~\ref{sec:main_result}.

Since $\Phi$ is weak$^*$-to-weak$^*$ continuous, it is the adjoint of an isomorphism 
\[
\mathcal{F}(\mathbb{H}^d) \simeq \mathcal{F}(P) \oplus \mathcal{F}(\mathcal{N}).
\]
Moreover, by the results of~\cite{CKK2016} $\mathcal{F}(P)$ is isomorphic to $\mathcal{F}(\mathbb{R}^{d})$, which yields 
\[
\mathcal{F}(\mathbb{H}^d) \simeq \mathcal{F}(\mathbb{R}^{d}) \oplus \mathcal{F}(\mathcal{N}) \qquad \text{and} \qquad \Lip_0(\mathbb{H}^d) \simeq \Lip_0(\mathbb{R}^d) \oplus \Lip_0(\mathcal{N})
\]
for $d=2,3,4$. Since both the space $\F(\R^d)$ and $\F(\mathcal{N})$ have a Schauder basis, the former result being due to P.~H\'ajek and E.~Perneck\'a in \cite{HP2014} and the latter to M.~Doucha and P.~Kaufmann in~\cite{DK2022}, we conclude that, for $d=2,3,4$, $\F(\H^d)$ also admits a Schauder basis, thereby giving a partial positive answer to the first question mentioned above. Let us recall that if one is aiming for weaker structural properties, such as the bounded approximation property, or the $(\pi)$-property, then more general results are available; see, for instance, \cite{GodefroyBJMA, LancienPernecka} and the references therein.

Using the same methods, we also show that the space $\Lip(\mathbb{H}^d)$ is isomorphic to $\Lip(\mathbb{R}^d)$, this time by giving an entirely explicit isomorphism (Remark~\ref{rem:Lip Hd}). Combining this result with standard arguments, in Remark~\ref{rem:Lip Sd} we conclude that 
\[
\Lip(\mathbb{H}^d) \simeq \Lip(\mathbb{R}^d) \simeq \Lip(\mathbb{S}^d),
\]
i.e. the spaces of bounded Lipschitz functions on the model spaces of metric geometry are all isomorphic.

At the core of our argument we have to decompose a Lipschitz function on $\mathbb{H}^d$ into a Lipschitz function on a net $\mathcal{N}$ and a sequence of Lipschitz functions on a convex subset of~$\mathbb{H}^d$. In order to do this, we consider a suitable tiling of $\mathbb{H}^d$ by polytopes; the existence of such tilings depends upon classical results from the theory of reflection groups, see e.g.~\cite[Chapter~6]{Davis} or~\cite[Chapter~5]{VS-GeoII}. More precisely, given a right-angled polytope $P$ (namely, such that all dihedral angles are exactly $\pi/2$), by reflecting across the faces of $P$ we obtain a tessellation of $\mathbb{H}^d$ by isometric copies of $P$. In~\cite{Vinberg} (see also \cite[Section~2]{PV}) the author shows that such right-angled polytopes exist only if $d\leq 4$; explicit constructions in dimensions $d=2,3,4$ were already known to exist (see Section~\ref{ssec:Tiling}). This justifies why we are able to prove our results only in dimension $d=2,3,4$. In order to emphasise the subtlety of hyperbolic tilings results, let us mention two more results: in dimension $d\geq6$ there exists no regular tiling of $\H^d$ \cite[p.~206]{Coxeter-honeycomb} and, more generally, for $d\geq 30$ there are no hyperbolic reflection groups at all (see e.g.~\cite[Theorem~6.11.8]{Davis}).

Given a tiling of $\H^d$ by right-angled polytopes we first use an extension operator from the net, given by choosing a distinguished point inside the polytope, to decompose a $\Lip_0$-function on $\H^d$ into a function on the net and a \emph{bounded} Lipschitz function on $\H^d$. Then, using an extension operator for Lipschitz functions on $P$, we decompose the bounded Lipschitz function in a bounded sequence of Lipschitz functions on $P$. The latter construction is inspired by a decomposition method for ($C^\infty$-)smooth functions on $\R^d$ into sequences of functions on the unit cube in~\cite{Bar14} and~\cite{Bar15}.

Let us close this section with a brief description of the structure of the paper. In Section~\ref{sec:preliminaries} we recall basic notions on Lipschitz-free spaces and metric geometry. A self contained revision of hyperbolic geometry is the content of Section~\ref{sec:hyperbolic}, in particular we explain the properties of the tilings that we need in Section~\ref{ssec:Tiling}. As we mentioned already, Section~\ref{sec:local result} is dedicated to the local problem and we study the space $\Lip_0(P)$, for a polytope $P$ in $\R^d$. Finally, the core of our paper with the proof of the main results is Section~\ref{sec:main_result}.

\section{Preliminary material}\label{sec:preliminaries}
Given a pointed metric space $(M,d)$ with distinguished point $0_M\in M$ we consider the Banach space $\Lip_0(M)$ of all Lipschitz functions $f\colon M\to \R$ such that $f(0_M)=0$, endowed with the norm
\[
\|f\|_{\Lip_0} \coloneqq \Lip(f) \coloneqq \sup \left\{\left| \frac{f(x)-f(y)}{d(x,y)}\right| \colon x\neq y\in M \right\}.
\]
Moreover, when $(M,d)$ is a metric space, we consider the vector space of all bounded Lipschitz functions $f\colon M\to \R$ that, following \cite[Chapter 2]{Weaver}, we denote by $\Lip(M)$. $\Lip(M)$ becomes a Banach space when equipped with the norm
\[
\|f\|_\Lip\coloneqq \|f\|_{\Lip_0}+ \|f\|_\infty.
\]
The pointwise multiplication induces an algebra structure on $\Lip(M)$ due to the basic inequality
\[
\Lip(fg)\leq \Lip(f) \|g\|_\infty + \Lip(g) \|f\|_\infty.
\]
When $M$ is bounded, the same product also gives an algebra structure on $\Lip_0(M)$, because $\|f\|_\infty \leq \diam(M) \|f\|_{\Lip_0}$. Actually, a different product turns every $\Lip_0(M)$ into a Banach algebra, \cite{AACD-bounded}.

For $p\in M$, the evaluation functional $\delta_p\in \Lip_0(M)^*$ is defined by $\langle\delta_p, f\rangle\coloneqq f(p)$. It is easy to see that $\|\delta_p\|= d(p,0_M)$. Then one can define $\F(M)\coloneqq \overline{\rm span}\{\delta_p\colon p\in M\}\subset \Lip_0(M)^*$ and verify that $\F(M)$ satisfies the universal property stated in the Introduction. In particular, $\F(M)^*=\Lip_0(M)$.
\smallskip

As mentioned in the Introduction, our argument will proceed in $\Lip_0(M)$ and only at the very end we will pass to preduals and deduce results for $\F(M)$. Therefore, we need information on the weak$^*$ topology of $\Lip_0(M)$ induced by the predual $\F(M)$. By definition, the set $\{\delta_p\colon p\in M\}$ of elementary molecules is linearly dense in $\F(M)$. Thus, on bounded sets, the weak$^*$ topology coincides with the weak topology induced by the functionals $\{\delta_p \colon p\in M\}$. In other words, it agrees with the topology of pointwise convergence on $M$. When combined with the Banach--Dieudonn\'e theorem, this fact has important consequences. First, a subspace $X\subset \Lip_0(M)$ is weak$^*$ closed if and only if it is pointwise closed. Second, a bounded operator $L\colon \Lip_0(M)\to \Lip_0(N)$ is \wtow continuous if and only if it is pointwise-to-pointwise continuous (see e.g. \cite[Exercise 3.65]{FHHMZ} or \cite[Corollary 2.33]{Weaver}). These facts will be freely used multiple times in our arguments. Moreover, if $X$ is a weak$^*$ closed subspace of $\Lip_0(M)$, then it is the dual to some quotient $Z$ of $\F(M)$ and the weak$^*$ topology of $X$ induced by $Z$ coincides with the restriction to $X$ of the weak$^*$ topology of $\Lip_0(M)$, see e.g. \cite[Corollary V.2.2]{Conway}.

An ubiquitous role in our proofs will be played by linear extension operators. If $N$ is a subset of $M$ with $0_M\in N$, a \emph{linear extension operator} $E\colon \Lip_0(N)\to \Lip_0(M)$ is a linear operator such that $Ef$ is an extension of $f$ for every $f\in \Lip_0(N)$. Plainly, the restriction operator $f\mapsto f|_N$ is a left-inverse to $E$. Thus, $E$ defines an isomorphic embedding of $\Lip_0(N)$ as a complemented subspace of $\Lip_0(M)$. Moreover, if $E$ is pointwise-to-pointwise continuous, both the isomorphic embedding and the projection are \wtow continuous. The construction of linear extension operators for Lipschitz functions has been a topic of great interest in recent years, see e.g. \cite{AACD, AP-extension, Brudnyi, BMS, Godefroy-Extensions, JLS, LeeNaor}.

In most cases, our extensions operators will be based on direct and explicit constructions. However, in one place we shall use the following deep extension result due to Lee and Naor \cite{LeeNaor}: there is a universal constant $C$ such that for every pointed metric space $(M,d)$ and every subspace $N$ with $0_M\in N$ that is $\lambda$-doubling there is a linear extension operator $E\colon \Lip_0(N)\to \Lip_0(M)$ with $\|E\|\leq C\log(\lambda)$. For a more recent simpler proof see \cite[Theorem 4.1]{BMS}; importantly, the operator $E$ is also \wtow continuous, as it is explained for example in \cite{LancienPernecka}.
\smallskip

Next, we need to mention some basics on metric geometry. Recall that a metric space $(M,d)$ is \emph{geodesic} if for every two points $x,y\in M$ there is an isometry $\gamma\colon [0,d(x,y)]\to M$ such that $\gamma(0)=x$ and $\gamma(d(x,y))=y$. In case such an isometry $\gamma$ is unique, $(M,d)$ is said to be \emph{uniquely geodesic}. Intuitively speaking, geodesics are the metric analogue of segments. Accordingly, the image $\gamma([0,d(x,y)])$ of a geodesic connecting $x$ to $y$ is sometimes called \emph{metric segment} and denoted $[xy]$. Moreover, we say that a subset $C$ of $M$ is \emph{(geodesically) convex} if every metric segment $[xy]$ with $x,y\in C$ is entirely contained in $C$. Let us refer e.g. to \cite{BH} for more on geodesic metric spaces.

An important property of geodesic metric spaces is that the Lipschitz condition becomes a local property, in the sense of the following lemma. For convex subsets of Banach spaces it is due to D.J.~Ives and D.~Preiss in~\cite{IP}.

\begin{lemma}[{\cite[Lemma~2.1]{BRT}}]\label{LipschitzLemma}
    Let $M$ be a geodesic space, $N$ a metric space, and $\{Z_i\}_{i=1}^{\infty}$ a countable family of sets covering $M$. Let $f\colon M\to N$ be a continuous mapping whose restrictions to the sets $Z_i$ are Lipschitz and satisfy $\Lip(f|_{Z_i})\leq L$ for some $L>0$. Then $f$ is Lipschitz with $\Lip(f) \leq L$.
\end{lemma}

\section{Hyperbolic geometry}\label{sec:hyperbolic}
This section is dedicated to a brief introduction to the \emph{hyperbolic $d$-space} $\H^d$. The shortest way to introduce it is to define $\H^d$ as the unique complete, simply-connected Riemannian $d$-manifold with constant sectional curvature $-1$. The uniqueness of a Riemannian $d$-manifold with such properties is a consequence of the Killing--Hopf theorem, see e.g. \cite[Theorem 12.4]{Lee}. However, for our purposes it will be more convenient to have an explicit description of a model for $\H^d$; we shall now recall two such models and later use whichever model is more convenient for our purpose.

Let us start by recalling the \emph{hyperboloid model}: consider the non-degenerate bilinear form
\[
\langle x,y\rangle\coloneqq \sum_{i=1}^d x_i y_i - x_{d+1}y_{d+1}
\]
on $\R^{d+1}$ and define
\[
\H^d\coloneqq \left\{x\in \R^{d+1}\colon \langle x,x\rangle = -1,\ x_{d+1}>0 \right\}.
\]
A metric on $\H^d$ can be defined by
\[
\rho(x,y)= \arcosh(- \langle x,y\rangle).
\]
As it turns out, $\langle \cdot,\cdot \rangle$ is positive definite on the tangent bundle of $\H^d$ and $\rho$ is exactly the Riemannian distance induced by the Riemannian metric $\langle \cdot,\cdot \rangle$ on $\H^d$. Geodesic lines are defined as the intersections of $\H^d$ with $2$-dimensional subspaces of $\R^{d+1}$. In particular, $\H^d$ is uniquely geodesic and every geodesic can be uniquely extended to a geodesic line. Angles are also defined in terms of $\langle \cdot,\cdot \rangle$: given two geodesics that meet at a point $\xi\in \H^d$, one takes unit tangent vectors $u$ and $v$ to the geodesics at $\xi$ and defines the angle between the geodesics as the unique $\alpha\in [0,\pi]$ such that $\cos \alpha= \langle u,v\rangle$. Note that, in particular, if the point in question is the origin, then $u_{d+1}v_{d+1}=0$, thus $\langle u,v\rangle_{\H^d}=\langle u,v\rangle_{\R^{d+1}}$, and thus the angle is the Euclidean one. For the same reason, if any geodesic subspace containing the origin meets a geodesic at an orthogonal angle in $\R^{d+1}$, then the (two possible) tangent vectors $u$ of said geodesic will satisfy $u_{d+1}=0$ and the hyperbolic angle will be orthogonal as well.

An \emph{isometry} of $\H^d$ is a bijection of $\H^d$ that preserves the distance $\rho$; it then follows that such isometries also preserve angles. Let us recall that the group of isometries acts transitively on $\H^d$; moreover, each isometry of $\H^d$ can be obtained as a composition of at most $d+1$ reflections through hyperplanes. More on isometries of $\H^d$ can be found in \cite[Chapter~19]{Berger}, \cite[Chapter~I.2]{BH}, or \cite[Section~10]{CFKP}.
\smallskip

We next briefly describe the \emph{Beltrami--Klein model}. It is represented by points in the Euclidean open unit ball $\mathbb{B}^d\coloneqq \left\{x\in \R^d\colon \|x\|_2< 1\right\}$. Geodesics are simply intersections of Euclidean lines with $\mathbb{B}^d$. Likewise, hyperplanes are intersections of Euclidean hyperplanes with $\mathbb{B}^d$. The metric for $\H^d_{BK}$ can be defined as follows: given points $x$ and $y$ in $\mathbb{B}^d$, let $x_\infty$ and $y_\infty$ be the intersections of the line through $x$ and $y$ with the boundary of $\mathbb{B}^d$ (arranged so that $x_\infty,x,y,y_\infty$ appear in order). Then
\[
\rho_{BK}(x,y)\coloneqq \frac{1}{2} \log\frac{\|x-y_\infty\|_2 \|y-x_\infty\|_2} {\|x-x_\infty\|_2 \|y-y_\infty\|_2}.
\]

Further information on this and more models can be found in \cite[Chapter~19]{Berger}, or \cite[Chapter~I.6]{BH}. Let us notice that, by the Killing--Hopf theorem all such models are mutually isometric. More importantly, simple and explicit isometries between the models are available, thus allowing explicit transfer of properties; see e.g. \cite[Chapter~19]{Berger}, or \cite[Theorem 3.7]{Lee}. Finally, for a gentle and elementary introduction to hyperbolic geometry, mainly in the plane, we refer the interested reader to~\cite{Anderson}.

The last result we require is a particular case of the well-known fact from Riemannian geometry that differentiable maps between Riemannian manifolds are locally Lipschitz. A direct, computational proof of this particular case can be given by using the explicit formula for the metric.
\begin{lemma}\label{lem:Beltrami-Klein-bilipschitz} Let $\H^d_{BK}$ be the Beltrami-Klein model of $\H^d$ and let $\mathbb{B}^d$ be the Euclidean open unit ball in $\R^d$. Then the identity function $\id\colon \H^d_{BK} \to \mathbb{B}^d$ is locally bi-Lipschitz.
\end{lemma}

\subsection{Tilings of \texorpdfstring{$\H^d$}{Hd}}\label{ssec:Tiling}
We call a subset of $\R^d$ or $\H^d$ a \emph{polyhedron} if it is a finite intersection of closed half-spaces. A bounded (or equivalently compact) polyhedron is called a \emph{polytope}. Note that polyhedra are convex sets, as is, for example, obvious in the Beltrami--Klein model. Hence the nearest point projection, namely the mapping which assigns to a point $x$ the point in the polyhedron which minimises the distance to $x$, is well-defined and $1$-Lipschitz, see Proposition~2.4 in~\cite[p.~176]{BH}. In particular, polyhedra are $1$-Lipschitz retracts of both $\R^d$ and $\H^d$.

In order to investigate the local and global structure of the Lipschitz functions on $\H^d$ separately, we will need a suitable tiling of the hyperbolic space $\H^d$. A sequence $(P_n)_{n\in\N}$ of polytopes is a \emph{tiling}, or \emph{tessellation}, of $\H^d$ if $\bigcup_{n\in\N}P_n= \H^d$ and the intersection of any two distinct polytopes is either empty or a face of both polytopes (in particular, the interiors of the polytopes $P_n$ are mutually disjoint).

In our arguments we will use the existence in $\H^d$ for $d=2,3,4$ of a \emph{regular orthogonal tiling}, namely a tiling $(P_n)_{n\in\N}$ consisting of mutually isometric polytopes and where each $P_n$ satisfies the following definitions (see e.g. \cite[Chapter~6]{Davis} for more details):
\begin{enumerate}[label=(T\arabic*),ref=T\arabic*]
    \item\label{T1-regular} A polytope $P$ in $\H^d$ is \emph{regular} if the group of isometries of $P$ is flag-transitive. More precisely, a \emph{flag} in $P$ is a chain $F_0\subset F_1\subset \dots\subset F_{d-1}$ where $F_0$ is a vertex of $P$ and each $F_k$ is a $k$-dimensional face of $P$. Then $P$ is regular if for every two flags $F_0\subset F_1\subset \dots\subset F_{d-1}$ and $F'_0\subset F'_1\subset \dots\subset F'_{d-1}$ there is an isometry of $P$ that maps one flag onto the other.
    \item\label{T2-orthogonal} A polytope is \emph{right-angled}, or \emph{orthogonal}, if all dihedral angles are exactly $\pi/2$.
\end{enumerate}

Before passing to the explanation of the existence of the tiling, let us mention two more properties that follow from \eqref{T1-regular} and \eqref{T2-orthogonal} (two further properties will be proved in Lemma~\ref{lem:orthogonal-reflections-commute} and Lemma~\ref{lem:tiles-form-hyperplanes} below).
\begin{enumerate}[label=(T\arabic*),ref=T\arabic*]\setcounter{enumi}{2}
    \item\label{T3-incentre} Every polytope $P_n$ has an inscribed circle, whose centre we denote by $p_n$ and whose radius is by definition the \emph{in-radius} of $P_n$.\\
    Indeed, every regular polytope $P$ admits a centre, whose existence can be shown as follows. If $F$ is a maximal face of $P$, then $x\mapsto \dist(x,F)$ is a convex function (as follows easily from Proposition~2.2 in \cite[p.~176]{BH}). Thus, letting $\{F_1,\dots,F_k\}$ be the maximal faces of $P$, the map $x\mapsto \dist(x,F_1)^2 +\dots+ \dist(x,F_k)^2$ is strictly convex and hence has a unique minimum $p$. Since the above map is defined only by metric properties, every isometry of $P$ must fix $p$. Finally, by regularity, $p$ has the same distance to all maximal faces. For further details we refer to \cite[pp.~178--179]{BH}.
    \item\label{T4-intersections} There is a number $N(d)$ such that every $P_n$ intersects at most $N(d)$ polytopes from the tiling.\\
    Indeed, at each vertex of $P_n$ there are exactly $2^d$ polytopes intersecting $P_n$, by the right-angle property. So, a (rough) upper bound for $N(d)$ is $2^d$ times the number of vertices of $P_n$ (such a number of vertices is independent of $n$, as the polytopes are mutually isometric).
\end{enumerate}

We now pass to explaining the existence of regular orthogonal tilings $(P_n)_{n\in\N}$ for $d=2,3,4$. By Proposition~6.3.2 and~6.3.9 in~\cite{Davis} every polytope $P$ whose dihedral angles are of the form $\pi/m$ for some integer $m\geq 2$ (these polytopes are sometimes called \emph{Coxeter polytopes}) is simple. Hence by Theorem~6.4.3 in~\cite{Davis} $P$ is a strict fundamental domain of the reflection group generated by reflections across the faces of $P$. In other words $P$ intersects every orbit in exactly one point. This implies in particular that the images of $P$ by the elements of the reflection group tessellate~$\H^d$. Thus, every element of the reflection group maps the faces of a polytope $P_n$ to the faces of some other polytope of the tiling.

If we restrict our attention to right-angled polytopes, by the above observations, the existence of a right-angled polytope $P$ directly implies that there is a tessellation of $\H^d$ by isometric copies of $P$. Several explicit constructions of such polytopes for dimensions $d=2,3,4$ are available in the literature and we mention a few of them in the next paragraph. On the other hand, E.B.~Vinberg proved in~\cite{Vinberg} that there is no right-angled polytope in $\H^d$ when $d\geq 5$; a short proof of this fact can also be found in~\cite[Section~2]{PV}. Consequently, an orthogonal tiling of $\H^d$ exists precisely for dimensions $d=2,3,4$.

As it turns out, in dimension $d=2,3,4$ an orthogonal tiling can be constructed that is additionally regular. In dimension $d=2$, there are infinitely many regular orthogonal tilings, one for each integer $p\geq5$. They are obtained by taking a regular right-angled $p$-gon in $\H^2$ and gluing together $4$ of them at each vertex; by means of the so-called \emph{Schl\"{a}fli symbol} such a configuration is denoted by $\{p,4\}$. For instance, $\{8, 4\}$ is a tiling via four regular octagons meeting at each vertex, illustrated in Figure~\ref{fig:octagonal_tiling}. In dimension $d=3$ a tiling exists described by $\{5,3,4\}$, which is to be read as follows: start with a regular right-angled pentagon in $\H^2$ and glue three of them at each vertex. This gives a (hyperbolic) dodecahedron $\{5,3\}$ and gluing $4$ on each edge gives the tiling. In dimension $d=4$, instead, one glues together $3$ dodecahedra each edge to obtain a $4$-dimensional polytope $\{5,3,3\}$ whose maximal faces are the dodecahedra $\{5,3\}$. Then, gluing $4$ such polytopes at each $2$-dimensional face one gets the tiling $\{5,3,3,4\}$. A list of all regular tilings of $\H^d$ can be found in~\cite[Section~5.3.3]{VS-GeoII} or~\cite[Sections~2--4]{Coxeter-honeycomb} (the orthogonal ones among them are exactly those with $4$ as the last number in their Schl\"{a}fli symbol). 
\smallskip

\begin{figure}
    \includegraphics[width=.6\textwidth]{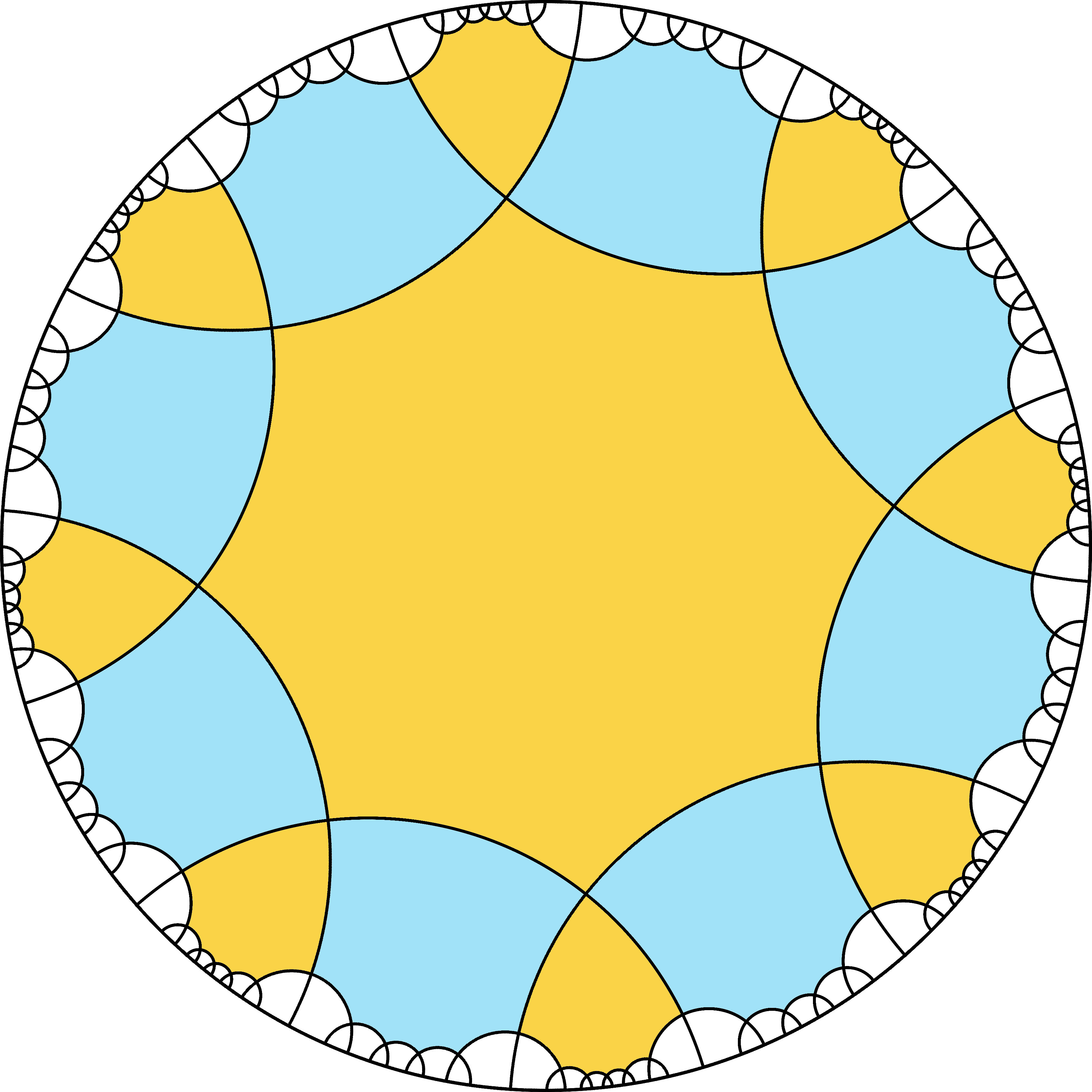}
    \caption{The first seventeen isometric octagons tiling the hyperbolic plane.\label{fig:octagonal_tiling}}
\end{figure}

One notable feature of orthogonal tilings, which we will use in the next lemma as well as Section~\ref{sec:main_result}, is the following:

\begin{lemma}\label{lem:orthogonal-reflections-commute}
    Whenever two hyperplanes $H_1, H_2\subset \H^d$ meet orthogonally, the reflections $R_{H_1}$ and $R_{H_2}$ across them commute and $R_{H_1} R_{H_2}= R_{H_2} R_{H_1}= R_{H_1\cap H_2}$. Further, $R_{H_2}[H_1] =H_1$ and $R_{H_2}$ preserves the two half-spaces defined by $H_1$.
\end{lemma}

\begin{proof} We will prove this statement using the hyperboloid model of $\H^d$. Since for every point in $\H^d$ there exists an automorphism mapping it to the origin, we can assume that $0\in H_1\cap H_2$. Moreover, up to a rotation around $0$ of $\H^d$ (i.e. a rotation of $\R^{d+1}$ around the $x_{d+1}$ axis), we may assume that $H_1=\{y\in\H^d\colon y_1=0\}$ and $H_2=\{y\in\H^d\colon y_2=0\}$ (recall that angles at the origin $0_{\H^d}$ are Euclidean angles).

Then it follows that $R_{H_1}(x) = (-x_1,x_2,\dots,x_{d+1})$ and $R_{H_2}(x)=(x_1,-x_2,\dots,x_{d+1})$ for $x\in \H^d$. Clearly these two mappings commute and their composition is the mapping $x\mapsto (-x_1,-x_2,x_3,\dots,x_{d+1})$, which is the reflection across $H_1\cap H_2$. The last clause also follows directly from these formulas.
\end{proof}

As one can already glean from Figure~\ref{fig:octagonal_tiling}, the edges of the octagons combine into geodesic lines, and this is true in general for any regular orthogonal tiling.

\begin{lemma}\label{lem:tiles-form-hyperplanes} Let $\mathcal{P}=(P_n)_{n\in\N}$ be a regular orthogonal tiling of $\H^d$, $d\leq4$. Then, there exists a sequence $(H_k)_{k\in\N}$ of hyperplanes in $\H^d$ such that
    \begin{align*}
        \bigcup_{n\in\N}\partial P_n = \bigcup_{k\in\N}H_k.
    \end{align*}
    These $H_k$ are exactly all the supporting hyperplanes for any maximal face of any $P_n$. As a consequence, if $R$ is a reflection across some maximal face,
    \[ R\left[ \bigcup_{k\in\N}H_k \right]= \bigcup_{k\in\N}H_k. \]
\end{lemma}

\begin{proof} Let $(H_k)_{k\in\N}$ be the sequence listing all the supporting hyperplanes for any maximal face of any polytope $P_n$. Then the validity of the `$\subset$' inclusion is clear.

For the converse inclusion, take a hyperplane $H\in\{H_k\}_{k\in\N}$. By definition, there is a polytope $P\in \mathcal{P}$ such that $M\coloneqq H\cap P$ is a maximal face of $P$. Then $M$ is a regular right-angled polytope in $H$ (which is isometric to $\H^{d-1}$), hence the reflection group generated by reflections across the maximal faces of $M$ induces a tiling of $H$, by copies of $M$. However, a reflection across a maximal face of $M$ is the restriction to $H$ of a reflection across a hyperplane $\tilde{H}$ in $\H^d$, orthogonal to $H$ (in fact, the reflection across $\tilde{H}$ maps $H$ to $H$, by Lemma~\ref{lem:orthogonal-reflections-commute}). Moreover, $\tilde{H}$ supports a maximal face of $P$, as $P$ is right-angled. Therefore, the images of $M$ under the reflection group of $M$ are images of maximal faces of $P$ under the reflection group of $P$. Hence, $H\subset \bigcup_{n\in\N}\partial P_n$, as desired.

Finally, the last claim follows because $R$ preserves the set $\bigcup_{n\in\N}\partial P_n$ of faces.
\end{proof}

\section{Lipschitz functions on polytopes}\label{sec:local result}
This section is dedicated to the local part of our construction, where we study the Banach space of Lipschitz functions on a single polytope. The result we prove in this section asserts that the space of all Lipschitz functions that vanish on certain subsets of the boundary of a polytope $P$ is \wtow isomorphic to $\Lip_0(P)$. Since $\H^d$ is locally bi-Lipschitz equivalent to $\R^d$ by Lemma~\ref{LipschitzLemma}, it is irrelevant if we consider Euclidean or hyperbolic polytopes; therefore, in all the section we consider polytopes in $\R^d$. As we explained in the Introduction, our focus being on the hyperbolic case, we treat the Euclidean case as a black-box; this explains why in this section we will use non-explicit arguments and we don't have an explicit formula for the isomorphism.

We begin by introducing the following notation.
\begin{definition} If $S$ is a subset of a pointed metric space $M$, we denote by
\[
\Lip_{0,S}(M):=\{f\in\Lip_0(M)\colon f|_S=0\},
\]
the space of all Lipschitz functions on $M$ vanishing on $S$ (in addition to vanishing on $0_M$).
\end{definition}

It is easy to see that $\Lip_{0,S}(M)$ is pointwise closed, hence weak$^*$ closed as well. Thus, $\Lip_{0,S}(M)$ is the dual of a quotient of $\F(M)$ and the corresponding weak$^*$ topology is the restriction of the weak$^*$ topology of $\Lip_0(M)$, hence on bounded sets it coincides with the pointwise one.

Before the main result of the section, we collect in the following lemma a basic construction of a linear extension operator.

\begin{lemma}\label{lem:ExtensionOperatorRetract}
Let $(M,d)$ be a pointed metric space and $N\subset M$ be a bounded Lipschitz retract of $M$ with $0_M\in N$. Then for every $\e>0$ there is a bounded, \wtow continuous, linear extension operator
\[
E_{N,\e} \colon \Lip_0(N) \to \Lip_0(M), \qquad f \mapsto E_{N,\e}f
\]
such that, for every $f\in \Lip_0(N)$, the support of $E_{N,\e}f$ is contained in the set $\overline{B}(N,\e)=\{x\in M\colon {\rm dist}(x,N)\leq \e\}$.
\end{lemma}

\begin{proof} Fix a Lipschitz retraction $r$ of $M$ onto $N$ and define the mapping $\lambda\colon M\to [0,1]$ by $\lambda(x) = \max\{1-\frac{ {\rm dist}(x,N)}{\e}, 0\}$. $\lambda$ is clearly $1/\e$-Lipschitz and vanishes outside $\overline{B}(N,\e)$. Then the mapping
\[
E_{N,\e}\colon \Lip_0(N) \to \Lip_0(M), \qquad f\mapsto (f\circ r)\lambda
\]
is well defined and bounded since
\begin{align*}
\Lip\left( (f\circ r)\lambda \right) &\leq \Lip(f\circ r) \|\lambda\|_\infty + \Lip(\lambda) \|f\circ r\|_\infty\\
&\leq \Lip(f)\Lip(r) +\frac{1}{\e} \|f\|_\infty \leq \left(\Lip(r) + \frac{\diam (N)}{\e}\right) \|f\|_{\Lip_0}.
\end{align*}
Finally, the pointwise-to-pointwise continuity of $E_{N,\e}$ is clear from the definition.
\end{proof}

\begin{proposition}\label{prop:polygons-with-and-without-edge-isomorphic}
Let $P$ be a polygon in $\R^d$ with $0\in {\rm int}(P)$ and let $S\subset \partial P$ be the union of one or more faces (of any dimension) of $P$. Then,
\[
\Lip_0(P)\simeq \Lip_{0,S}(P).
\]
Moreover, the isomorphism is \wtow continuous.
\end{proposition}

\begin{proof} Let us first prove that $\Lip_0(P)$ is \wtow isomorphic to a weak$^*$-com\-ple\-men\-ted subspace of~$\Lip_{0,S}(P)$. For this, note that, for $\alpha\in(0,1)$, $\Lip_0(P)$ and $\Lip_0(\alpha P)$ are \wtow isomorphic ($P$ and $\alpha P$ are bi-Lipschitz equivalent). Hence, it is enough to prove that $\Lip_0(\alpha P)$ is \wtow isomorphic to a weak$^*$-complemented subspace of~$\Lip_{0,S}(P)$. Now, $\alpha P$ is a Lipschitz retract of $P$ (via the nearest point projection) and it has positive distance to the boundary $\partial P$ of $P$. So, by Lemma~\ref{lem:ExtensionOperatorRetract} for $\e ={\rm dist} (\partial P,\alpha P)/2>0$ there is a bounded, \wtow continuous, linear extension operator
\[
E_{\alpha P,\e}\colon \Lip_0(\alpha P) \to \Lip_0(P)
\]
such that $E_{\alpha P,\e} f(x)=0$ for all $x\in \partial P$ (so, the image of $E_{\alpha P,\e}$ is contained in $\Lip_{0,S}(P)$). As we observed in Section~\ref{sec:preliminaries}, this gives the desired embedding of $\Lip_0(\alpha P)$ into $\Lip_{0,S}(P)$.

Next, we show that $\Lip_{0,S}(P)$ is a weak$^*$-complemented subspace of $\Lip_0(P)$. Consider the subset $S\cup\{0\}$ of $P$ and observe that $P$ is a doubling metric space (as the doubling property passes to subspaces and $\R^d$ is doubling). Therefore, we can apply Lee's and Naor's extension result \cite{LeeNaor} (see Section~\ref{sec:preliminaries}) and find a \wtow continuous linear extension operator $E\colon \Lip_0(S\cup\{0\})\to \Lip_0(P)$. Then the map
\[
Q\colon \Lip_0(P)\to\Lip_{0,S}(P)\colon f\mapsto f - E(f|_{S\cup\{0\}})
\]
is a linear projection onto $\Lip_{0,S}(P)$ since by definition $Qf$ vanishes on $S$ and for each $f\in\Lip_{0,S}(P)$, $E(f|_{S\cup\{0\}})=0$. Hence $\Lip_{0,S}(P)$ is a complemented subspace of $\Lip_0(P)$. In addition, $Q$ is \wtow continuous, since $E$ is.
	
Finally, by standard duality, the assertions proved in the previous paragraphs yield that $\Lip_{0,S}(P)$ is the dual to a complemented subspace $Z$ of $\F(P)$ and $\F(P)$ is isomorphic to a complemented subspace of $Z$. Moreover, $P$ has nonempty interior in $\R^d$, hence combining \cite[Corollary~3.5]{K2014} and \cite[Theorem~3.1]{K2014} we obtain that $\F(P)$ is isomorphic to its $\ell_1$-sum. Therefore, Pe{\l}czy\'{n}ski decomposition method assures us that $\F(P)$ is isomorphic to $Z$ and passing to the duals we reach the sought conclusion.
\end{proof}

\section{Extension operators and proof of the main result}\label{sec:main_result}
In this section, we construct a number of extension operators which are the central tools for the proof of the main result. We are using a regular orthogonal tiling and hence we consider $\H^d$ only for $d=2,3,4$. Given a regular orthogonal tiling $\mathcal P:=(P_n)_{n\in\N}$  by isometric copies of a single polytope, we denote by $p_n$ the centre point of $P_n$ in the sense of~\eqref{T3-incentre} of Section~\ref{ssec:Tiling} and consider the net $\mathcal{N}:=\{p_n\colon n\in\N\}$. For the space~$\Lip_0(P_n)$ we use $p_n$ as the distinguished point and for the sake of simplicity, we assume that $p_1=0$.

Using these data we use the following strategy to prove the main result: In Section~\ref{ssec:NetExtension} we use an extension operator from the net~$\mathcal{N}$ to decompose Lipschitz functions on~$\H^d$ into Lipschitz functions on~$\mathcal{N}$ and (bounded) Lipschitz functions vanishing on $\mathcal{N}$. In Section~\ref{ssec:Decomposition} we decompose the latter functions into a sequence of Lipschitz functions on the tiles. In Section~\ref{ssec:ExtensionFromTiles} we use extension operators from the tiles to construct the inverse operator of this decomposition operator. Finally, in Section~\ref{ssc:MainProofEnd} we combine these arguments to finish the proof and state some consequences.

\subsection{Extension from the net \texorpdfstring{$\mathcal N$}{N}}\label{ssec:NetExtension}
In this part we construct a bounded, \wtow continuous linear extension operator from the net $\mathcal{N}$ to $\H^d$. In order to achieve this, we will exploit a Lipschitz partition of unity in the spirit of \cite{AACD, BMS, JLS}. Notice that said results cannot be applied directly, because $\H^d$ is not a doubling metric space. On the other hand, the specific shape of the net, inherited from the tiling, allows a very simple construction of the partition of unity.

Since the polytopes $P_n$ are mutually isometric, the in-radius ${\rm dist}(p_n, \partial P_n)$ of $P_n$ is the same for all $n$; thus we can set $\delta\coloneqq {\rm dist}(p_n, \partial P_n)>0$. Therefore, when $k\neq n$, ${\rm dist}(p_k,P_n)\geq \delta$. Define functions $\varrho_n\colon \H^d \to \R$ by $\varrho_n(x)\coloneqq \max\left\{1-\frac{{\rm dist}(x,P_n)}{\delta},0 \right\}$. Then $\varrho_n(p_k)= \delta_{k,n}$ and $\varrho_n$ is $1/\delta$-Lipschitz. Moreover, the restriction of $\varrho_n$ to $P_k$ is nonzero only when $P_k$ intersects $P_n$. Hence, by \eqref{T4-intersections} the series $\sum_{n=1}^\infty \varrho_n$ is locally a sum of at most $N(d)$ terms and therefore it defines a bounded Lipschitz function. Additionally, we have $\sum_{n=1}^\infty \varrho_n(x)\geq 1$ for all $x\in \H^d$, as $\varrho_n$ equals $1$ on $P_n$. Consequently, the functions
\[
\p_n\coloneqq \frac{\varrho_n}{\sum_{n=1}^\infty \varrho_n}
\]
are uniformly Lipschitz, with constant, say, $L_\mathcal{N}$. Clearly, $(\p_n)_{n=1}^\infty$ is the desired partition of unity and $\p_n(p_k)=\delta_{k,n}$. We can now proceed to the construction of the extension operator from the net $\mathcal{N}$.

\begin{lemma}\label{lem:extension-operator-net} The operator $E_{\mathcal N}$ defined as
\begin{align*}
    E_{\mathcal N}\colon \Lip_0(\mathcal N)&\to \Lip_0(\H^d)\colon
    f\mapsto \sum_{n=1}^\infty f(p_n)\varphi_n
\end{align*}
is a bounded linear extension operator from the net $\mathcal N$ to $\H^d$ which is \wtow continuous.
\end{lemma}

\begin{proof}
    The sum defining $E_{\mathcal N}f$ is locally finite, thus $E_{\mathcal N}$ is clearly a well-defined linear extension operator and $E_\mathcal{N}f$ is continuous. In order to bound the Lipschitz constant of $E_{\mathcal N}f$, we use Lemma~\ref{LipschitzLemma}. Thus, fix $f\in \Lip_0(\mathcal N)$ and $x,y\in P_n$, for some $n\in \N$. Let $\{n_1,\dots, n_{N(d)}\}$ be the set of indices corresponding to the polytopes $P_k$ that intersect $P_n$. Hence we have
    \begin{align*}
    |E_\mathcal{N}f(x) - E_\mathcal{N}f(y)| =&
    \left|\sum_{j=1}^{N(d)} \big(\p_{n_j}(x) - \p_{n_j}(y)\big) f(p_{n_j})\right|\\
    =& \left|\sum_{j=1}^{N(d)}  \big(\p_{n_j}(x) - \p_{n_j}(y)\big) \big(f(p_{n_j})-f(p_n)\big)\right|\\
    \leq& \sum_{j=1}^{N(d)} L_\mathcal{N}\ \rho(x,y)\ \Lip(f)\ \rho(p_{n_j},p_n)\leq C\ \Lip(f)\ \rho(x,y).
    \end{align*}
    Here, $C= 2\diam(P_n)\ L_\mathcal{N}\ N(d)$ and we used the fact that $\rho(p_{n_j},p_n)\leq 2\diam(P_n)$, since $P_{n_j}\cap P_n\neq\emptyset$. According to Lemma~\ref{LipschitzLemma}, this inequality proves that $\Lip(E_\mathcal{N}f)\leq C\Lip (f)$, whence the boundedness of $E_\mathcal{N}$. 

    Lastly, to prove \wtow continuity, we need only show pointwise-to-pointwise continuity. So assume that $(f_k)_{k\in\N}$ is a sequence of Lipschitz functions converging pointwise to some function $f\in\Lip_0(\mathcal N)$. Then, using again local finiteness of the sum,
    \begin{align*}
        \lim_{k\to\infty}E_{\mathcal N}(f_k)(x) &= \sum_{n=1}^\infty \lim_{k\to\infty}f_k(p_n) \p_n(x) = \sum_{n=1}^\infty f(p_n) \p_n(x) = E_{\mathcal N}(f)(x).\qedhere
    \end{align*}
\end{proof}

\subsection{Decomposition into Functions on Tiles}\label{ssec:Decomposition}
The aim of this section is to decompose a Lipschitz function on $\H^d$ into a sequence of functions defined on the tiles of our tiling. As a first step, we use the results of the previous section to remove the part of the function defined on the net~$\mathcal{N}$. 
More precisely, Lemma~\ref{lem:extension-operator-net} implies that the mapping
\[
\Lip_0(\H^d)\to \Lip_0(\mathcal{N}) \oplus \Lip_{0,\mathcal{N}} (\H^d)\colon \qquad f \mapsto \big(f|_\mathcal{N}, f-E_{\mathcal{N}}(f|_\mathcal{N}) \big)
\]
is a \wtow continuous isomorphism (with inverse $(g,h)\mapsto E_\mathcal{N}(g) + h$). 

Hence, our goal now is to decompose a function in~$\Lip_{0,\mathcal{N}}(\H^d)$ into a sequence of functions defined on the tiles. We start with some geometric preliminaries.

\begin{definition}
    Let $H$ be a hyperplane in $\H^d$ not containing the origin~$0$. The two connected components of $\H^d\setminus H$ are called the \emph{open half-spaces} defined by $H$. We denote their closures by $H^+$ and $H^-$ where $H^+$ is chosen so that $0$ is in the interior of~$H^+$.
\end{definition}

Let us now fix some notation and parameters. Let $\mathcal P$ be our tiling of $\H^d$ and $(H_m)_{m\in\N}$ be the sequence of hyperplanes supporting the faces of the tiles $P_n$ as in Lemma~\ref{lem:tiles-form-hyperplanes}. Let $\e>0$ be a fixed number which is smaller than the in-radius of $P_m$ and smaller than half of the minimal distance of non-adjacent faces of $P_m$. For the hyperplane $H_m$, we denote by $R_m\colon \H^d\to \H^d$ the reflection across $H_m$. 

In order to start the construction of the decomposition, for each hyperplane $H_n$ we define a cutoff function
\[
\psi_n\colon \H^d\to [0,1]\colon \qquad x\mapsto \max\left\{1-\frac{ {\rm dist}(x,H_n)}{\e},0\right\}
\]
which clearly satisfies $\Lip(\psi_n)\leq 1/\e$, $\psi_n(x)=1$ if and only if $x\in H_n$, and $\psi_n(x)=0$ if and only if $\dist(x,H_n)\geq \e$. 

We now construct operators on the space $(\bigoplus_{m\in\N}\Lip_0(P_m))_{\ell_\infty}$. In order to simplify the notation, we often interpret this sum as the space of all functions $g$ defined on $\bigcup_{m\in\N} {\rm int}(P_m)\subset \H^d$ which are Lipschitz on every ${\rm int}(P_m)$ with $\sup_{m\in \N} \Lip(g|_{P_m})< \infty$ and that vanish on the net~$\mathcal{N}$. More precisely, to a sequence $(g_m)_{m\in \N}$ we associate the function $g$ defined to be equal to $g_m$ on ${\rm int}(P_m)$, for every $m\in\N$; vice-versa, to a function $g$ we associate the sequence $(g_m)_{m\in \N}$, where $g_m$ is the unique continuous extension of $g|_{{\rm int} (P_m)}$ to $P_m$. The important advantage of this interpretation is that we do not consider values on the boundaries of the tiles, which simplifies several formulas. As a first instance of this, this interpretation allows us to identify $\Lip_{0,\mathcal{N}}(\H^d)$ with the subspace of all functions admitting a continuous extension to~$\H^d$.

With these preparations, we are now able to define the operator
\[
\chi_n\colon \Big(\bigoplus_{m\in\N}\Lip_0(P_m)\Big)_{\ell_\infty} \to \Big(\bigoplus_{m\in\N}\Lip_0(P_m)\Big)_{\ell_\infty}
\]
by setting
\[
(\chi_n g)(x) = \begin{cases} g(x) & \text{for}\;x\in H_n^{-}, \\ g(x) - \psi_n(x) g(R_n(x)) &\text{for}\; x\in H_n^+.
\end{cases}
\]
Notice that in the definition of $\chi_n$ we used the interpretation explained above. Also, according to Lemma~\ref{lem:tiles-form-hyperplanes}, $R_n$ maps $\bigcup_{m\in\N} {\rm int}(P_m)$ to itself, thus we can legitimately evaluate $g$ at the point $R_n(x)$.

Since reflections are isometries and $g$ is bounded, the operator $\chi_n$ is bounded and it is clearly also \wtow continuous. Also note that since $\chi_n$ does not change the values of~$g$ on~$H_n^-$ and $R_n(x)\in H_n^-$ whenever $x\in H_n^+$, its inverse is the operator defined by
\[
(\chi_n^{-1} g)(x) = \begin{cases} g(x) & \text{for}\;x\in H_n^{-}, \\ g(x) + \psi_n(x) g(R_n(x)) &\text{for}\; x\in H_n^+.
\end{cases}
\]
We denote by $\chi_{1,\ldots,n}:=\chi_n\circ\cdots\circ \chi_1$ the composition of $\chi_1,\ldots, \chi_n$ and note that it satisfies $\chi_{1,\ldots,n}^{-1} = \chi_1^{-1}\circ\cdots\circ\chi_n^{-1}$.

Given a function $g\in \Lip_{0,\mathcal N}(\H^d)$, we define a function $g_m\colon P_m\to \R$ by
\[
    g_m(x) := \lim_{n\to\infty} \chi_{1,\ldots,n}(g)(x) \qquad (x\in P_m).
\]
Note that for each $m$ there is an $N_m$ such that $\chi_{1,\ldots,n}(g)(x)$ does not change for $n\geq N_m$. Hence, the above limit exists since the sequence is eventually constant and the index where it becomes constant only depends on $P_m$ and not on the individual point $x\in P_m$.

Let us now consider the linear mapping 
\[
\Lip_{0,\mathcal N}(\H^d)\to \Big(\bigoplus_{m\in\N}\Lip_0(P_m)\Big)_{\ell_\infty}\colon \qquad g\mapsto (g_m)_{m\in\N}
\]
which will eventually be the main ingredient for our isomorphism. Since it is not surjective between the above spaces, we have to determine a suitable codomain. With this in mind, for each $m\in \N$ we define
\begin{equation}\label{eq:Sm}
S_m:=\bigcup_{\substack{n\in\N\\P_m\subset H_n^+}}H_n\cap P_m.    
\end{equation}
Now we are able to define the desired operator
\[
\Phi\colon \Lip_{0,\mathcal N}(\H^d)\to \Big(\bigoplus_{m\in\N}\Lip_{0,S_m}(P_m)\Big)_{\ell_\infty}\colon g \mapsto (g_m)_{m\in\N}   
\]
and start checking that it has the desired properties.

\begin{lemma}\label{lem:Phi well def} The linear operator $\Phi$ maps $\Lip_{0,\mathcal N}(\H^d)$ into $\big(\bigoplus_{m\in\N}\Lip_{0,S_m}(P_m)\big)_{\ell_\infty}$. Moreover, it is bounded and \wtow continuous.
\end{lemma}

\begin{proof}
    We first check that $\Phi(g)$ is a sequence of Lipschitz functions whose Lipschitz constants are bounded by a multiple of the Lipschitz constant of $g$. Since $g\circ R_n$ and $\psi_n$ are Lipschitz and bounded, the standard product formula gives us that $\chi_n$ is bounded with $\|\chi_n\|= \|\chi_1\|$ for all $n$. Moreover, $\chi_n$ only changes its argument on an $\e$-neighbourhood of $H_n$, and $P_m$ intersects said neighbourhood if and only if $H_n$ supports a face of $P_m$. Hence, the number of faces $N(d)$ of the polytopes $P_m$ gives us the number of times the function $g$ can at most be changed on that tile during the limit process that produces $\Phi(g)$. Thus, the Lipschitz constant of $\Phi(g)_m$ is at most 
    \[
    \Lip(\Phi(g)_m)\leq \|\chi_1\|^{N(d)} \Lip(g).
    \]
    Also, $\Phi(g)(p_m)=g(p_m)=0$ because $p_m\in\mathcal N$ and the centre-point $p_m$ of any $P_m$ is not in an $\e$-neighbourhood of any $H_n$. It follows that $\Phi(g)_m\in\Lip_0(P_m)$. Note that the above inequality also shows the boundedness of $\Phi$. By the same reason, the \wtow continuity of $\chi_n$ implies that $\Phi$ is also \wtow continuous.

    To conclude, we need to check that $(\Phi(g))_m$ vanishes on $S_m$. For this aim, assume $x\in P_m\cap H_n$ and $P_m\subset H_n^+$ and observe that
    \[
    \chi_{1,\ldots,n}(g)(x)= \chi_{1,\ldots,n-1}(g)(x) - \underset{=1}{\underbrace{\psi_n(x)}}(\chi_{1,\ldots,n-1}(g)(x))(\underset{=x}{\underbrace{R_nx}}) = 0.
    \]
    So at the $n$-th step of our limit process, the values of our function $\chi_{1,\dots,n}g$ are set to zero everywhere along $H_n$ on the tiles that lie within $H_n^+$. To express this briefly, we say that $\chi_{1,\dots,n}g$ vanishes on $\partial H_n^+$. We now have to show that this value of zero is retained by all the subsequent functions $\chi_{1,\dots,k}g$, $k>n$. We will do this by induction.

    So, fix $k>n$ and assume that $\chi_{1,\dots,k-1}g(x)=0$ for $x\in \partial H_n^+$. To show that $\chi_{1,\dots,k}g(x)=0$ as well, we distinguish two cases. If $H_k$ and $H_n$ are not orthogonal, it follows from our choice of $\e$ that $\psi_k$ is zero on $H_n$. Hence, $\chi_{1,\dots,k}g(x)=\chi_{1,\dots,k-1}g(x)=0$ for $x\in \partial H_n^+$.
    On the other hand, if $H_k$ is orthogonal to $H_n$, Lemma~\ref{lem:orthogonal-reflections-commute} implies that $R_k[H_n]=H_n$ and $R_k[H_n^+]= H_n^+$. Therefore, $\chi_{1,\dots,k-1}g(R_k x)=0$ when $x\in \partial H_n^+$. As before, it follows that $\chi_{1,\dots,k}g$ vanishes on $\partial H_n^+$, which finishes the proof.
\end{proof}

In the following subsection we will construct an inverse operator for $\Phi$, thus showing that $\Phi$ is an isomorphism.

\subsection{Extension Operators from the Tiles and inverse of \texorpdfstring{$\Phi$}{Phi}}\label{ssec:ExtensionFromTiles}
The goal of this section is to construct a linear extension operator from $\Lip_{0,S_m}(P_m)$ to $\Lip_{0,\mathcal{N}}(\H^d)$ which works well with the decomposition considered in the previous section and allows us to show that the operator $\Phi$ is an isomorphism. The interpretation of $(\bigoplus_{k\in\N}\Lip_{0}(P_k))_{\ell_\infty}$ and its subspace $(\bigoplus_{k\in\N}\Lip_{0,S_k}(P_k))_{\ell_\infty}$ considered in the previous section allows us to view the space $\Lip_{0,S_m}(P_m)$ as a subset of $(\bigoplus_{k\in\N}\Lip_{0,S_k}(P_k))_{\ell_\infty}$, by setting the function equal to zero outside of $P_m$. 

For the construction of the extension operator for $\Lip_{0,S_m}(P_m)$, let $n_1>\dotsb>n_s$ be the natural numbers such that the hyperplane $H_{n_j}$ supports a face of $P_m$ which is not contained in $S_m$. As a first step we define the operator
\[
\Lip_{0,S_m}(P_m)\to\left(\bigoplus_{m\in\N}\Lip_{0}(P_m)\right)_{\ell_\infty}\colon \qquad h \mapsto(\chi_{n_s}^{-1}\circ \dots\circ \chi_{n_1}^{-1})(h).
\]
The next lemma shows that $(\chi_{n_s}^{-1}\circ \dots\circ \chi_{n_1}^{-1})(h)$ admits a continuous extension to~$\H^d$ and hence we may define the extension operator
\[
E_m\colon \Lip_{0,S_m}(P_m)\to \Lip_{0,\mathcal{N}}(\H^d)\colon \qquad h\mapsto E_mh.
\]
where $E_mh$ is the unique continuous extension of~$(\chi_{n_s}^{-1}\circ \dots\circ \chi_{n_1}^{-1})(h)$ to~$\H^d$.

\begin{lemma}\label{lem:E_m-well-defined}
    The operator $E_m$ is well-defined, bounded, linear, and \wtow continuous. Moreover, $\|E_m\|$ is bounded by a constant depending only on $\e$, the number of faces of $P_m$ and the diameter of $P_m$.
\end{lemma}

\begin{proof}
    In order to show that $E_m$ is well-defined, we have to show that~$g:=(\chi_{n_s}^{-1}\dots\chi_{n_1}^{-1})(h)$ admits a continuous extension to~$\H^d$. We show this in an inductive manner. Note that no hyperplane $H_n$ intersects the interior of $P_m$ and, by definition of~$S_m$ see~\eqref{eq:Sm}, $H_{n_j}^{+}$ is the half-space separated from $P_m$ by $H_{n_j}$ (otherwise $H_{n_j}$ would support a face contained in $S_m$, which it does not by definition of $n_1,\dots ,n_s$).
    
    Since we think of~$C_k:=\bigcup_{j=k+1}^{s} \mathrm{int}(H_{n_j}^{+})$ as the set to where we have not yet extended our function in the $k$-th step, we now show that $g_k:=(\chi_{n_{k}}^{-1}\circ \dots\circ \chi_{n_{1}}^{-1})(h)$ has an extension to $\H^d$ which is zero on $C_k$ and continuous on $\H^d\setminus C_k$. Since $g=g_s$ and $C_s=\emptyset$ this will prove our claim.
    
    The function $g_0=h$ is zero outside of $P_m$ and hence on $C_0$ and $(\H^d\setminus C_0)\setminus P_m$. Since the (relative) boundary of $P_m$ in the set $\H^d\setminus C_0$ is precisely $S_m$ and $h_m$ vanishes there, $g_0$ has a continuous extension to $\H^d\setminus C_0$.
    
    Let us assume now that we have already shown that $g_{k-1}$ has an extension to $\H^d$ which is zero on $C_{k-1}$ and continuous on $\H^d\setminus C_{k-1}$. By abuse of notation we call this extension $g_{k-1}$ as well. We now have to check that $g_{k}$ also has the desired properties.

    We first check that $g_k$ vanishes on $C_k$ and recall that $g_k(x)=g_{k-1}(x)$ for $x\in H_{n_k}^{-}$ and $g_{k}(x) = g_{k-1}(x) + \psi_n(x) g_{k-1}(R_{n_k}x)$ for $x\in H_{n_k}^+$. We distinguish between two cases: If for $j>k$ we have $H_{n_j}\cap H_{n_{k}}\neq \emptyset$, these hyperplanes are orthogonal and hence $x\in H_{n_j}^+$ if and only if $R_{n_{k}}x \in H_{n_j}^+$ and hence $g_k(x)=0$ for $x\in \mathrm{int}\,(H_{n_j}^+)$. For the other case note that $H_{n_{k}}$ cannot be in $H_{n_j}^+$ since otherwise the face of $P_m$ contained in $H_{n_{k}}$ would have to be in $H_{n_j}^+$ which is only possible if it was in $H_{n_j}$ which contradicts $H_{n_{k}}\cap H_{n_j}=\emptyset$. Hence, we have $\mathrm{dist}\,(H_{n_j}^+, H_{n_{k}})>2\varepsilon$ and therefore $g_k(x) = g_{k-1}(x)=0$ for $x\in H_{n_{s-k+1}}$.
   
    In order to show that $g_k$ has a continuous extension to $\H^d\setminus C_k$ let $z\in H_{n_k}\setminus C_k$ and note that $g_{k-1}(x)=0$ for $x\in\mathrm{int}(H_{n_k}^{-})$. Hence 
    \[
    \lim_{x\to z} g_{k}(x) = \lim_{x\to z} \psi_{n_k}(x) g_{k-1}(R_{n_k}(x)) = g_{k-1}(z)
    \]
    where $x$ converges to $z$ in the interior of $H_{n_k}^{-}$ since $\psi_{n_k}(z)=1$ and $R_{n_k}(z)=z$. This finishes the proof of the well-definedness of $E_m$.

    Linearity and pointwise-to-pointwise continuity are clear. Since $\chi_{n_s}^{-1}\dots\chi_{n_1}^{-1}$ is the composition of bounded operators on $(\bigoplus_{m\in\N}\Lip_{0}(P_m))_{\ell_\infty}$ and $E_mh_m$ is
    continuous, boundedness of $E_m$ follows from Lemma~\ref{LipschitzLemma} and $\|E_m\|\leq \|\chi_{n_s}^{-1}\circ \dots\circ \chi_{n_1}^{-1}\|$. Finally note that this norm only depends on $\e$, $s$ and the diameter of $P_m$ and that $s$ is at most the number of faces of~$P_m$.
\end{proof}

\begin{remark}\label{rem:ObservationEm}
    Note that applying $\chi_{n}^{-1}$ for $n$ distinct from $n_1,\ldots,n_s$ in definition of $E_m$ does not change anything as the function we are working with is zero on the interior of $H_n^{-}$, so instead of using $\chi_{n_s}^{-1}\circ\cdots\circ \chi_{n_1}^{-1}$ we can also use $\chi_{1,\ldots,n_1}^{-1}$ or even $\chi_{1,\ldots,k}^{-1}$ for every $k\geq n_1$.
\end{remark}

We are now in position to define the operator
\[
\Psi\colon \Big(\bigoplus_{m\in\N}\Lip_{0,S_m}(P_m)\Big)_{\ell_\infty} \longrightarrow \Lip_{0,\mathcal N}(\H^d)\colon \qquad (h_m)_{m\in\N}\mapsto \sum_{m\in\N}E_mh_m
\]
(where the sum is meant to be taken pointwise) and show that it is the inverse of $\Phi$.

\begin{proposition}\label{prop:PhiIsIsomorphism}
    $\Psi$ is well-defined, linear, bounded, \wtow continuous and the inverse of $\Phi$.
\end{proposition}

\begin{proof} The sum in the definition of $\Psi$ is locally finite with a uniform bound $N(d)$ on the number of summands; thus $\Psi$ is a bounded linear operator by Lemma~\ref{lem:E_m-well-defined}. Since \wtow continuity of $\Psi$ is easy to see, we are left to show that $\Psi$ is the inverse of $\Phi$.

We start by showing that $\Phi\circ\Psi = \mathrm{Id}$. For this, let $(h_m)_{m\in\N} \in (\bigoplus_{m\in\N}\Lip_{0,S_m}(P_m))_{\ell_\infty}$ and fix an arbitrary $k\in\N$. It is enough to check that $(\Phi(\Psi((h_m)_{m\in\N})))_k(x) = h_k(x)$ for $x$ in the interior of~$P_k$. Recall that in Section~\ref{ssec:Decomposition} we observed that there is an index $N_k\in\N$ such that $(\chi_{1,\ldots,n}g)(x)=(\chi_{1,\dots,N_k}g)(x)$ for $n\geq N_k$, every $g\in \Lip_{0,\mathcal{N}}(\H^d)$ and all $x\in P_k$. Using this observation together with the definition of these operators and Remark~\ref{rem:ObservationEm} we have
\begin{align*}
    (\Phi(\Psi((h_m)_{m\in\N})))_k(x) &= \Big(\Phi\Big(\sum_{m\in\N} E_mh_m\Big)\Big)_k(x) = \sum_{m\in\N}\chi_{1,\ldots,N_k}(E_mh_m)(x) \\
    &= \sum_{m\in\N} (\chi_{1,\ldots,N_k} (\chi_{1,\ldots,N_k}^{-1}(h_m))) (x) = \sum_{m\in\N} h_m(x) = h_k(x)
\end{align*}
for all $x\in \mathrm{int}\,(P_k)$, i.e. $(\Phi(\Psi((h_m)_{m\in\N})))_k=h_k$.

In order to show that $\Psi\circ\Phi = \mathrm{Id}$, let $g\in\Lip_{0,\mathcal{N}}(\H^d)$ be given. It is enough to check that $\Psi\circ\Phi (g)(x)= g(x)$ for all $x$ in the interior of $P_k$ and for all $k$. Since $\chi_n^{-1}((\Phi(g))_k)(x)$ for $k\neq n$ is only nonzero if $x\in H_{n}^+$ and within $\e$-distance of $H_n$ and there are at most $d$ hyperplanes $H_{n_1},\ldots, H_{n_d}$ with this property there are numbers $M,N\in\N$ such that 
\begin{align*}
    (\Psi(\Phi(g)))(x) &= \sum_{m=1}^{M} (E_{m}((\Phi(g))_{m}))(x) = \sum_{m=1}^{M} (\chi_{1,\ldots,N}^{-1})(\Phi(g))_{m})(x)\\
    &= \chi_{1,\ldots,N}^{-1}\Big(\sum_{m=1}^{M} \mathbbb_{\mathrm{int}\,P_{m}} \chi_{1,\ldots,N}(g)\Big)(x) \\
    &= \chi_{1,\ldots,N}^{-1}\chi_{1,\ldots,N}\Big(\sum_{m=1}^{M} \mathbbb_{\mathrm{int}\,P_{m}} g\Big)(x) = g(x)
\end{align*}
since we may interpret the sum $\sum_{m=1}^{M} \mathbbb_{\mathrm{int}\,P_{m}} \chi_{1,\ldots,N}(g)$ locally around $x$ as a representation of $\chi_{1,\ldots,N}(g)$ as a function similar to the previous sections and the operator $\chi_{1,\ldots, N}$ and its inverse are defined pointwise and only depend on `nearby' values of the function.
\end{proof}

\subsection{Conclusion of the proof and consequences}\label{ssc:MainProofEnd}
In this section we use the arguments of the previous sections and combine them with the results of Section~\ref{sec:local result} to prove our main result. Moreover we state a number of direct consequences of the main result and compare it to the case of the space of bounded Lipschitz functions $\Lip(\H^d)$.

\begin{theorem}
   For $d=2,3,4$ we have 
   \[
   \Lip_0(\H^d)\simeq\Big(\Lip_0(\mathcal N) \oplus \bigoplus_{m\in \N} \Lip_0(P_m)\Big)_{\ell_\infty}
   \]
   and 
   \[
   \F(\H^d)\simeq \Big(\F(\mathcal N) \oplus \bigoplus_{m\in \N} \F(P_m)\Big)_{\ell_1}.
   \]
\end{theorem}

\begin{proof}
    We consider the mapping
    \[
    \Lip_0(\H^d) \to \Lip_0(\mathcal{N}) \oplus \Big(\bigoplus_{m\in\N}\Lip_{0,S_m}(P_m)\Big)_{\ell_\infty}\colon \qquad f\mapsto \big(f|_{\mathcal{N}}, \Phi(f-f|_\mathcal{N})\big)
    \]
    which by Lemma~\ref{lem:extension-operator-net}, Lemma \ref{lem:Phi well def}, and Proposition~\ref{prop:PhiIsIsomorphism} is a \wtow continuous isomorphism. By Lemma~\ref{lem:Beltrami-Klein-bilipschitz}, $P_m$ is bi-Lipschitz equivalent to the same polytope $P_m'$ considered in $\R^d$. Combining this with Proposition~\ref{prop:polygons-with-and-without-edge-isomorphic} implies that $\Lip_{0,S_m}(P_m)$ and $\Lip_0(P_m)$ are \wtow isomorphic, with a uniform bound on the distortion (because the polytopes $P_m$ are mutually isometric). This gives the first of the claims above and shows that the isomorphisms are also \wtow continuous. Hence we may pass to the preduals and arrive at the second claim.
\end{proof}

\begin{corollary} Let $d=2,3,4$ and $\mathcal{M}$ be any net in $\H^d$. Then $\F(\H^d)$ is isomorphic to $\F(\mathcal{M})\oplus \F(\R^d)$.
\end{corollary}

\begin{proof}
    As proved by Bogopolskii in~\cite{Bo}, all nets in $\H^d$ are bi-Lipschitz equivalent, thus $\F(\mathcal M)\simeq\F(\mathcal N)$, where $\mathcal N$ is the net from the above theorem. Moreover, by Lemma~\ref{lem:Beltrami-Klein-bilipschitz} $P_1$ is bi-Lipschitz equivalent to a polytope in $\R^d$ (with non-empty interior), so $\F(P_1)\simeq\F(\R^d)$ due to~\cite[Corollary~3.5]{K2014}. Finally, each $\F(P_n)$ is isometric to $\F(P_1)$, thus $(\bigoplus_n \F(P_n))_{\ell_1} \simeq (\bigoplus_n \F(\R^d))_{\ell_1} \simeq F(\R^d)$, where the last isomorphism follows from~\cite[Theorem~3.1]{K2014}.
\end{proof}

\begin{remark} Alternatively, instead of~\cite{Bo}, we could have used the following result from~\cite{HN17}: 
if $\mathcal{N}$ and $\mathcal{M}$ are nets in a metric space, both of cardinality the density of the metric space, then
$\F(\mathcal{N})\simeq \F(\mathcal{M})$.
\end{remark}

\begin{corollary}
    For $d=2,3,4$, $\F(\H^d)$ has a Schauder basis.
\end{corollary}

\begin{proof}
    We will show that $\F(\mathcal{N})\oplus \F(\R^d)$ has a Schauder basis. $\F(\mathcal{N})$ does as proved by Doucha and Kaufmann in~\cite{DK2022}, whereas $\F(\mathbb{R}^d)$ does by~\cite{HP2014}. Thus, their direct sum has a Schauder basis as well.
\end{proof}

\begin{remark}\label{rem:Lip Hd}
Note that a similar construction using a tessellation with cubes works in $\mathbb{R}^d$. Additionally, if we look at the spaces of \emph{bounded} Lipschitz functions $\Lip(\R^d)$ and $\Lip(\H^d)$ rather than the ones which are zero at a given base point, then we can use an analogous construction to the one we employ in Section~\ref{ssec:Decomposition} and Section~\ref{ssec:ExtensionFromTiles} to decompose any function $f\in\Lip(\H^d)$ and any $g\in\Lip(\R^d)$ into a sequence of bounded Lipschitz functions $(f_n)_{n\in\N}$ and $(g_n)_{n\in\N}$ on the tiles $(P_n)_{n\in\N}$ tessellating $\H^d$ and the cubes $(Q_n)_{n\in\N}$ tesselating $\R^d$, respectively. Since the functions in $\Lip(\mathbb{R}^{d})$ and $\Lip(\mathbb{H}^d)$ are already bounded and functions in $\Lip(M)$ do not require a base point within the domain on which they are zero, removing the values on a net is not needed in this case. Thus, for $d=2,3,4$,
\begin{gather*}
\Lip(\mathbb{H}^d) \simeq \Big( \bigoplus_{n\in \mathbb{N}} \Lip_{S_n}(P_n)\Big)_{\ell_\infty} \qquad\mbox{and}\qquad \Lip(\mathbb{R}^d) \simeq \Big(\bigoplus_{n\in \mathbb{N}} \Lip_{T_n}(Q_n) \Big)_{\ell_\infty},
\end{gather*}
where $S_n$ is defined as in Definition \ref{eq:Sm} and $T_n$, analogously, is the union of all boundary (hyper-)surfaces of the (hyper-)cubes $Q_n$ which are supported by a hyperplane which does not separate $Q_n$ from the origin $0_{\R^d}$. Equivalently, the union of all faces of $Q_n$ which contain the point(s) of $Q_n$ that is (are) furthest from the origin.

From this we are able to obtain that $\Lip(\mathbb{H}^d)\simeq \Lip(\mathbb{R}^d)$, this time in a completely explicit manner.

Indeed, since $S_n$ is the part of the boundary of $P_n$ which is invisible from the origin, i.e. the geodesics connecting a point from $S_n$ with the origin intersect the polytope in other points, an elementary geometric argument shows that it is simply connected. Similarly, $T_n$ is a simply connected subset of $Q_n$. It follows that one can find explicit bi-Lipschitz mappings between any pair $(P_n,S_n)$ and $(Q_m,T_m)$. The only two tiles for which this is not true are the two tiles $P_1\subset \H^d$ and $Q_1\subset\R^d$ which contain their spaces' respective base points $0$, since those are the only two tiles whose `boundary condition' $S_n$ is the entire boundary $\partial P_1$ and $\partial Q_1$, respectively. However, this simply means that we map $P_1$ to $Q_1$ in a bi-Lipschitz way, and then carry on with the rest of the sequence arbitrarily since any other $(P_n,S_n)$ can be bi-Lipschitz mapped to any, say, $(Q_m,T_m)$.

Finally, there are only finitely many different pairs $(P_n,S_n)$, which are not pairwise congruent (and similarly finitely many pairs $(Q_m,T_m)$). Therefore, these bi-Lipschitz maps have uniformly bounded distortion, hence induce and isomorphism $\Lip(\mathbb{H}^d)\simeq \Lip(\mathbb{R}^d)$.
\end{remark}

\begin{remark}\label{rem:Lip Sd}
The isomorphism $\Lip(\mathbb{S}^d)\simeq\Lip(\R^d)$ can be deduced from the corresponding isomorphism for the $\Lip_0$-spaces. First note that since $\mathbb{S}^d$ has a finite diameter the space $\Lip_0(\mathbb{S}^{d})$ is a hyperplane in $\Lip(\mathbb{S}^d)$. Moreover we have
\[
\Lip_0(\mathbb{S}^d) \simeq \Lip_0(\R^d) \simeq \Lip_0([0,1]^d)
\]
by~\cite{AACD} or~\cite{FG2022}. As for the sphere, the space $\Lip_0([0,1]^d)$ is a hyperplane in $\Lip([0,1]^d)$. By \cite[Theorem~5]{CDW2016} all the $\Lip_0$-spaces above and hence also all spaces of bounded Lipschitz functions contain a (complemented) copy of $\ell_\infty$; hence, they are isomorphic to their hyperplanes. Therefore we may conclude that $\Lip(\mathbb{S}^{d}) \simeq \Lip([0,1]^d)$. Using Proposition~\ref{prop:polygons-with-and-without-edge-isomorphic} together with an argument similar to the one in Remark~\ref{rem:Lip Hd}, we obtain that $\Lip(\mathbb{S}^{d})\simeq \Lip(\R^d)$. In contrast to the case of the hyperbolic space, this argument works for arbitrary~$d$.
\end{remark}

\textbf{Acknowledgements.} The authors wish to thank Professors Martin Pfurner and Hans-Peter Schr\"{o}cker for interesting discussions on geometric aspects of this article. We are also grateful to the anonymous referee for the helpful report. The research of C.~Bargetz and F.~Luggin was supported by the Austrian Science Fund~(FWF): I~4570-N. F.~Luggin gratefully acknowledges the support by the Vice-Rectorate for Research at the University of Innsbruck in the form of a PhD scholarship.


\end{document}